\documentclass[reqno,a4paper]{amsart}
\usepackage[utf8]{inputenc}
\usepackage[T1]{fontenc}
\usepackage{amssymb,amsthm,amsmath,mathtools}
\usepackage[hypertexnames=false,hidelinks]{hyperref}
\usepackage{xcolor}
\usepackage[all]{xy}

\usepackage{mathrsfs}

\theoremstyle{plain}
\newtheorem{theorem}[subsection]{Theorem}
\newtheorem{proposition}[subsection]{Proposition}
	
\newtheorem{lemma}[subsection]{Lemma}
\newtheorem{conjecture}{Conjecture}

\theoremstyle{remark}

\newtheorem{remark}[subsection]{Remark}

\newcommand{\noproof}{\hfill \qed}

\numberwithin{equation}{section}

\newcommand{\CC}{\mathbb{C}}
\newcommand{\C}{\mathcal{C}}
\newcommand{\LL}{\mathcal{L}}
\newcommand{\LLL}{\mathbb{L}}
\newcommand{\M}{\mathcal{M}}
\newcommand{\CCC}{\mathbb{C}[\mathcal{C}_4]}
\newcommand{\CCCn}{\mathbb{C}[\mathcal{C}_n]}
\newcommand{\CCV}{\mathbb{C}[\Com_4]}
\newcommand{\CCVn}{\mathbb{C}[\Com_n]}
\newcommand{\git}{\mathord{/\mkern-6mu/}}

\newcommand{\ring}{\CC[a_1, a_2, a_3, a_5, a_6, a_9, a_{10}, a_{14}]}
\newcommand{\Com}{Com}

\DeclareMathOperator{\Tr}{Tr}

\DeclareMathOperator{\GL}{GL}

\DeclareMathOperator{\Sym}{Sym}
\DeclareMathOperator{\Var}{Var}

\begin{document}
	
	\title[On the coordinate ring of the fourth Calogero-Moser space]{On the coordinate rings of Calogero-Moser spaces and the invariant commuting variety of \\ a pair of matrices}

	\author{F.~Eshmatov}
	\author{X.~García-Martínez}
	\author{Z.~Normatov}
	\author{R.~Turdibaev}
	
	\email{f.eshmatov@newuu.uz}
	\email{xabier.garcia.martinez@uvigo.gal}
	\email{z.normatov@mathinst.uz}
	\email{r.turdibaev@newuu.uz}

	\address[Farkhod Eshmatov]{New Uzbekistan University, Department of Mathematics, 
		100007 Tashkent, Uzbekistan}
	\address[Xabier García-Martínez]{CITMAga \& Universidade de Vigo, Departamento de Matemáticas, Esc.\ Sup.\ de Enx.\ Informática, Campus de Ourense, E--32004 Ourense, Spain}
	\address[Rustam Turdibaev]{CITMAga \& Universidade de Santiago de Compostela, Departamento de Matemáticas, Rúa Lope Gómez de Marzoa, s/n, 15782 Santiago de Compostela, Spain}
	\address[Zafar Normatov]{School of Mathematics, Jilin University, Changchun, 130012, China}

	\thanks{This work was supported by Agencia Estatal de Investigaci\'on de Espa\~{n}a (Spain, European FEDER support included), grants PID2020-115155GB-I00 and PID2021-127075NA-I00, and by Xunta de Galicia through the Competitive Reference Groups (GRC), ED431C 2023/31.}

	\begin{abstract} This paper presents a comprehensive description of the coordinate rings and Poisson brackets associated with the fourth Calogero-Moser space and invariant commuting pairs of matrices of size four. As an application, we compute their respective classes in the Grothendieck ring of the category of complex varieties and we offer some novel insights about the geometry of the Hilbert scheme of points on the affine plane.  
		
	\end{abstract}
	
	
	\maketitle
	
	\section{Introduction}

	Let $\M_{n}$ be the set of all $n\times n$ complex matrices. The general linear group of degree~$n$ acts on the set of $d$-copies of $\M_{n}$ by the simultaneous conjugation, i.e., for any $g \in \GL_n(\CC)$,
	\begin{equation}\label{general_generator}
		g \cdot (X_1, \dots, X_d) = (gX_1g^{-1}, \dots, gX_dg^{-1}).
	\end{equation}
	It consequently induces an action of $\GL_n$ on the algebra $\CC[\M_n^{d}]$ of polynomial functions on $\M_n^d$. A celebrated result of Procesi~\cite{Pr} and Razmyslov \cite{Ra} states that the algebra of~ $\GL_n$-invariant polynomials $\CC[\M_n^d]^{\GL_n}$  is generated by the trace functions $\Tr(A_1\cdots A_k)$, where $A_1,\dots,A_k\in\{X_1, \dots, X_d\}$
	and every relation between them follows from the Cayley-Hamilton theorem. 
	
	For an integer $n\geq 0$, we introduce the following two subspaces of $\mathcal{M}_{n} \times \mathcal{M}_{n}$
	\begin{equation}
		\label{cmrel}
		\mathrm{C}_n\coloneqq \{ (X,Y) \in  \mathcal{M}_{n} \times \mathcal{M}_{n}
		\mid \mathrm{rank}([X,Y] +I_n)=1\}\, .
	\end{equation}
	and 
	\begin{equation}
		\label{comrel}
		\mathrm{Com}_n\coloneqq \{ (X,Y) \in  \mathcal{M}_{n} \times \mathcal{M}_{n}
		\mid  [X,Y]=0\}\, ,
	\end{equation}
	The action \eqref{general_generator} restricts to these subspaces so we define the $n$-th 
	\textit{Calogero-Moser space} $\mathcal{C}_n$ and the $n$-th 
	\textit{invariant commuting variety} ${\mathcal Com}_n$ to be the respective quotient varieties 
	\[
	\mathcal{C}_n\coloneqq \mathrm{C}_n \git \GL_n , \qquad {\mathcal Com}_n\coloneqq \mathrm{Com}_n \git \GL_n 
	\]
	These spaces hold significant importance in various branches of mathematics, including algebraic geometry (Hilbert schemes), representation theory (double affine Hecke algebras), deformation theory (symplectic reflection algebras), and Poisson geometry. As such, it is of great interest to have an explicit description of their coordinate rings, $\mathbb{C}[\mathcal{C}_n]$ and $\mathbb{C}[{\mathcal Com}_n]$. The primary objective of the paper is to provide such a description for the case when $n = 4$. 
	
	Let us commence by reviewing what is currently known about these rings for a general $n$ and their explicit descriptions for $n=2, 3$.
	The natural embeddings of~$\mathcal{C}_n$ and ${\mathcal Com}_n$ into $(\mathcal{M}_{n}\times \mathcal{M}_{n}) \git \mathrm{GL}_n$ indicate that the rings
	$\mathbb{C}[\mathcal{C}_n]$ and $\mathbb{C}[{\mathcal Com}_n]$
	are quotients of $\CC[\M_n\times \M_n]^{\GL_n}$ and therefore, by the Procesi-Razmyslov result, they are generated by trace functions $\Tr(A_1 \cdots A_k)$, where each $A_i$ can either be~$X$ or $Y$. Due to the commutativity, the ring $\mathbb{C}[{\mathcal Com}_n]$ is generated by traces of the form $\Tr(X^iY^j)$, $1~\le~i,j\le n$. In fact, the same holds for $\mathbb{C}[\mathcal{C}_n]$. 
	Indeed, in Wilson's landmark work~\cite{W}, he proved that the space $\mathcal{C}_n$ is a smooth, irreducible, and complex symplectic variety with dimension of $2n$. To this end, he gave the following description of $\mathcal{C}_n$. Let us consider the space of all quadruples~$(X,Y,v,w)$ where $X, Y \in \M_n$  and
	$v$ and $w$ are column and row vectors of length $n$, respectively.
	Let $\widetilde{\C}_n$ be its subspace
	consisting of all $(X,Y,v,w)$ satisfying
	\begin{equation}\label{vector-covector}
		XY-YX+I_n = v w .
	\end{equation}
	Then the well-defined action of $\GL_n$  on $\widetilde{\C}_n$ given by 
	\[
	g \cdot (X,Y,v,w)=(gXg^{-1}, gYg^{-1}, gv, wg^{-1})
	\]
	is free and it yields a natural identification:
	$ \C_n \cong \widetilde{\C}_n /  \GL_n$.
	Now, using relation~\eqref{vector-covector},
	one can easily  show (see e.g.~\cite{CEET}) that the traces 
	$\Tr(X^iY^j), 1\le i,j\le n$ generate~$\mathbb{C}[\mathcal{C}_n]$.
	However, as in the description of the ring of invariants of pairs of matrices, it is more convenient to consider the correspondent traceless versions of the matrices $(X,Y)$:
	\begin{equation*}
		A\coloneqq X-\frac1n \Tr(X)I_n, \qquad B\coloneqq Y-\frac1n\Tr(B)I_n.
	\end{equation*}
	
	Procesi’s results \cite[Section 5]{Pr} provide a compelling motivation for considering traceless matrices, which considerably helps to ease the complexity of the computations. Specifically, there exists an isomorphism of \(\mathbb{Z}^d\)-graded algebras:  
	\[
	\CC[\M_n^d]^{\GL_n} \cong \mathbb{C}[u_1,u_2,\dots,u_d] \otimes \CC[\M_n(0)^d]^{\GL_n},
	\]  
	where \(\CC[\M_n(0)^{d}]^{\GL_n}\) denotes the algebra of \(\GL_n\)-invariant polynomial functions on the space of \(d\)-tuples of traceless \(n \times n\) matrices. Notably, the elements \(u_i\) correspond to the traces of generic matrices. Applied to our case, this corresponds to the factorizations into products 
	\begin{equation}
		\label{factor}
		\mathcal{C}_n =\CC^2 \times \mathcal{C}_n/ \CC^2 \ , \quad {\mathcal Com}_n= \CC^2 \times {\mathcal Com}_n/ \CC^2 \, .
	\end{equation}

	In \cite[Proposition 11.35]{EG} and \cite{KMN}, it is established that the generators of $\mathbb{C}[\mathcal{C}_2]$ and  $\mathbb{C}[{\mathcal Com}_2]$ are 
	\[
	a_1\coloneqq\Tr(X),\, a_2\coloneqq\Tr(Y),\, a_3\coloneqq\Tr(A^2),\, a_4\coloneqq\Tr(AB),\, a_5\coloneqq\Tr(B^2),
	\]
	and their presentations are
	\begin{eqnarray*}
		\mathbb{C}[\mathcal{C}_2] &\cong& \mathbb{C}[a_1,a_2] \otimes \mathbb{C}[a_3,a_4,a_5]/(a_4^2-a_3a_5-1),\\ [0.2cm]
		\mathbb{C}[{\mathcal Com}_2] &\cong& \mathbb{C}[a_1,a_2]\otimes \mathbb{C}[a_3,a_4,a_5]/(a_4^2-a_3a_5) , 
	\end{eqnarray*}
	where $ \mathbb{C}[a_1,a_2]$ is the coordinate ring of $\CC^2$ in \eqref{factor}.
	
	As for the case $n=3$, the problem is fully solved in~\cite{NT-TG}. Explicitly, they have proved that both $\CC[\mathcal{C}_3]$ and $\mathbb{C}[{\mathcal Com}_3]$ are generated by  $a_1,a_2,\dots,a_9$, where
	\[
	a_6\coloneqq\Tr(A^3),  a_7\coloneqq\Tr(A^2B),  a_8\coloneqq\Tr(AB^2),  a_9\coloneqq\Tr(B^3) \, ,
	\]
	subject to the following relations. For any, $v \in \CC$, let $I_v$ be the ideal of $\CC[a_3, \dots ,a_9] $ generated 
	\begin{align*}\label{CM_3_relations}
		r_{1}&=a_3a_9-2a_4a_8+a_5a_7,\\
		r_{2}&=a_5a_6-2a_4a_7+a_3a_8,\\
		r_{3,v}&=9va_3-a_3a_4^2+a_3^2a_5+6a_6a_8-6a_7^2,\\
		r_{4,v}&=9va_4-a_4^3+a_3a_4a_5+3a_6a_9-3a_7a_8,\\
		r_{5,v}&=9va_5-a_4^2a_5+a_3a_5^2+6a_7a_9-6a_8^2.
	\end{align*}
	It was shown that $I_v \cong I_1$ if $v\neq 0$, and 
	\[
	\CC[\mathcal{C}_3] \cong \CC[a_1,a_2] \otimes \CC[a_3, \dots ,a_9]/ I_1 \ , \qquad \mathbb{C}[{\mathcal Com}_3] \cong \CC[a_1,a_2] \otimes \CC[a_3, \dots ,a_9]/ I_0.
	\]    
	An alternative presentation of \(\mathbb{C}[\mathcal{C}_3]\) is implicitly found by setting $d=3$ in the work of Bonnaf\'e, who derived it using a different set of generators in \cite[Theorem~3.5]{Bon}.
	
	We shall focus now in the case of $n=4$.
	It is known that the large collection of generators of~$\CC[\M_4\times \M_4]^{\GL_4}$  can be narrowed down to a minimal set formed by only~$32$, catalogued in~\cite{Dr}. However, the full description of the ideal of relations is yet to be found (see~\cite{Dr, DS}). Most of these $32$ generators contain traces of expressions involving the commutators. Hence, if we confine our pairs of matrices to the invariant commuting variety, the set of generators becomes constrained to $a_1, \dots a_{14}$, where~$a_1, \dots ,a_9$ as above and the rest are given by
	\begin{equation*}\label{generators}
		a_{10}\coloneqq\Tr(A^4), a_{11}\coloneqq\Tr(A^3B), a_{12}\coloneqq\Tr(A^2B^2), 
		a_{13}\coloneqq\Tr(AB^3),  a_{14}\coloneqq\Tr(B^4). 	
	\end{equation*}

	In~\cite{N}, the author showed that the same
	fourteen trace functions generate $\CCC$. Therefore, both~$\CCC$ and~$\mathbb{C}[{\mathcal Com}_4]$ can be expressed as quotients of the polynomial ring $\CC[a_1, \dots, a_{14}]$. Our attention will now turn to capturing and identifying the ideals of relations $I(\mathcal{C}_4)$ and $I({\mathcal Com}_4)$ in $\CC[a_1, \dots, a_{14}]$. 
	
	The main result in our paper is to find explicit descriptions of these ideals. 
	To this end, we use the Poisson brackets on the rings $\CCC$ and $\mathbb{C}[{\mathcal Com}_4]$, which arise from the symplectic structures on 
	$\mathcal{C}_4$ and on ${\mathcal Com}_4$ respectively. Let us briefly describe this symplectic structure. First, the space  $\mathcal{M}_n^2$ of the pair of matrices
	can be naturally identified with the cotangent bundle of $\mathcal{M}_n$, and the symplectic form defined by:
	\begin{equation}
		\label{symppair}
		\omega((X_1,X_2),(Y_1,Y_2))\coloneqq\Tr(X_1Y_2)-\Tr(X_2Y_1) \, .  
	\end{equation}
	Moreover, the moment map $\mu\colon \mathcal{M}_n^2 \to \mathfrak{sl}_n$ is given by $(X,Y)\mapsto XY-YX$, where $\mathfrak{sl}_n$ is the subspace of traceless matrices. Let $O\subset \mathfrak{sl}_n$ be a closed $\GL_n$-orbit with respect to the adjoint action. Then the orbit space $\mathfrak{M}_O\coloneqq\mu^{-1}(O)\git\GL_n$ is called an \textit{affine quiver variety} (see \cite{Nak}). The symplectic structure on $\mathcal{M}_n^2$ induces, via the Marsden-Weinstein symplectic reduction construction~\cite{MW}, a canonical symplectic structure on $\mathfrak{M}_O$. Note that both $\mathcal{C}_n \cong \mathfrak{M}_{O_2}$ and ${\mathcal Com}_n \cong \mathfrak{M}_{O_1}$ can be realized as quiver varieties. Here $O_1=\{0\}$ 
	and $O_2$ is the set consisting of $n\times n$
	matrices of the from $S-I_n$, where $S$ is a rank $1$ semi-simple matrix such that $\Tr(S)=n$. 
	Symplectic structures on $\mathcal{C}_n$ and ${\mathcal Com}_n$ yield a Poisson algebra structure on their corresponding coordinate rings.  Note that these rings are quotients of the polynomial algebra
	$B=\mathbb{C}[a_1,\dots,a_{n^2+2n}]$, where the generators are exactly ${ \Tr(A^iB^j)}$, with ${0\leq i,j\leq n }$.
	Hence using the Poisson structure on $\mathbb{C}[\mathcal{C}_n]$ (or/and $\mathbb{C}[{\mathcal Com}_n]$), we can define on~$B$  nonassociative brackets $\{-,-\}$, which is a derivation on the first and second components. This bracket will play an essential role in our construction.
	
	For $n=4$, as we discussed above, the number of minimal generators is $14$, and 
	the bracket $\{ - , -\}$ is defined on $\CC[a_1, \dots, a_{14}]$. Using simple linear algebra tools, we obtain a defining relation in $\CCC$:
	\[
	r_1=8a_3+a_3(a_4^2-a_3a_5)-2(a_7^2-a_6a_8)+2(a_3a_{12}-2a_4a_{11}+a_5a_{10}).
	\]
	Let  $I$  be the  subideal of $I(\mathcal{C}_4)$ in $\mathbb{C}[a_3,...,a_{14}]$ generated by $r_1$ via $\{-,-\}$ and the multiplication in the ring. Similarly, using as defining relation $r_1-8a_3$ in ${\mathcal Com}_4$, we let $J$ to be the subideal of  $I({\mathcal Com}_4)$ generated by $r_1-8a_3$. In this way, we obtain our main result:
	
	\begin{theorem}
		\label{mainThm}
		We have $I=I(\mathcal{C}_4)$ and $J=I({\mathcal Com}_4)$. In particular, there are Poisson algebra isomorphisms
		\[
		\CCC \cong  \CC[a_1,a_2] \otimes \dfrac{\CC[a_3,\dots, a_{14}]}{I} \, , \] 
		\[ \mathbb{C}[{\mathcal Com}_4] \cong 
		\CC[a_1,a_2]\otimes \dfrac{\CC[a_3,\dots, a_{14}]}{J} \, .
		\] 
	\end{theorem}
	For $n = 1, \dots , 4$, both $I(\mathcal{C}_n)$ and $I({\mathcal Com}_n)$ are generated by a single relation obtained through the Cayley-Hamilton theorem or the Jacobi identity. We propose the following conjecture for any $n\ge 1$.
	\begin{conjecture}\label{conjecture} Let $k=\frac{n(n+3)}{2}$. Then 
		\[
		\C[\mathcal{C}_n] \cong  \CC[a_1,a_2] \otimes \dfrac{\CC[a_3,\dots, a_{k}]}{I} \, , \] 
		\[ \C[{\mathcal Com}_n] \cong 
		\CC[a_1,a_2]\otimes \dfrac{\CC[a_3,\dots, a_{k}]}{J} \, , 
		\] 
		where $I$ and $J$ are generated via the bracket $\{-,-\}$ by a single relation which can naturally be obtained using the Cayley-Hamilton theorem on corresponding varieties.
	\end{conjecture}

	Recently, Bonnaf\'e and Thiel proposed an elegant algorithm for computing a presentation of the coordinate ring \(\mathbb{C}[\mathcal{C}_n]\) (see \cite[Theorem 3.15]{BT}). They use a different set of generators, which may yield a presentation that differs from the one in the above conjecture. However, Procesi pointed out to the fourth author, in private communication, that the Poisson bracket in \(\mathbb{C}[{\mathcal Com}_n]\) can be obtained via a type of Chevalley restriction theorem (see~\eqref{VacDom}) from the canonical Poisson bracket on \(\mathbb{C}[x_1, \ldots, x_n, y_1, \ldots, y_n]\). We hope that this fact may lead us to a proof of the above conjecture.

	Let \({\mathcal Var}_\C\) be the category of quasiprojective complex varieties, and let \(K_0({\mathcal Var}_\C)\) be its Grothendieck ring. Using the presentation in Theorem~\ref{mainThm}, we can stratify both~\(\mathcal{C}_4\) and \({\mathcal Com}_4\) to compute their classes in \(K_0({\mathcal Var}_\C)\). Explicitly, we obtain:
	\[
	[{\mathcal Com}_4] = \LLL^8, \quad [\mathcal{C}_4] = \LLL^8 - \LLL^7 + 2 \LLL^6-\LLL^5.
	\]
	
	Now we would like to discuss how our result can be applied to the geometry of Hilbert schemes. Recall that the $n$-th Hilbert scheme of points on the plane, denoted as $\mathrm{Hilb}_n(\mathbb{C}^2)$, is defined as $\widetilde{H}\git\mathrm{GL}_n$, where 
	\[
	\widetilde{H}\coloneqq\{(X,Y,i)\in \mathcal{M}_{n} \times \mathcal{M}_{n} \times \mathbb{C}^n \mid XY=YX ,\quad \mathbb{C}[X,Y]i= \mathbb{C}^n\} \, ,
	\]
	and this is acted upon by $\mathrm{GL}_n$ as $g\cdot(X,Y,i)=(gXg^{-1}, gYg^{-1}, gi)$. Moreover, in~\cite[Theorem 8.13]{Nak} an irreducible highest weight representation of a product of the Heisenberg and Clifford algebras on the direct sum of cohomologies $\oplus_{n} H^{\bullet}( \mathrm{Hilb}_n(\mathbb{C}^2))$ has been constructed. Although the construction is clear, its complexity calls for a desire to find a simpler proof. Additionally, in~\cite{Nak} it is indicated that the characters of representations of affine Lie algebras have certain modularity properties. These features of affine Lie algebras and modularity suggest connections to loops and elliptic curves. It is natural to expect that these objects would appear in $\mathrm{Hilb}_n(\mathbb{C}^2)$, and that the above construction of the highest weight representation would involve them. However, no elliptic curves or loops have been found in $\mathrm{Hilb}_n(\mathbb{C}^2)$, and the problem remains unsolved.

	Wilson and Nakajima independently demonstrated in their respective works \cite{W} and \cite{Nak} that $\mathrm{Hilb}_n(\mathbb{C}^2)$ and $\mathcal{C}_n$ are naturally diffeomorphic. This diffeomorphism allows us to establish an equivalence between their cohomology rings:
	\[
	\bigoplus_{n\ge 0} H^{\bullet}(\mathrm{Hilb}_n(\mathbb{C}^2)) \cong  \bigoplus_{n \ge 0} H^{\bullet} (\mathcal{C}_n) 
	\]
	Consequently, we can investigate Nakajima's construction on the Calogero-Moser spaces. Since $\mathcal{C}_n$ is an irreducible affine algebraic variety, its singular cohomology coincides with its de Rham cohomology $\mathrm{H}^{\bullet}_{dR}(\mathcal{C}_n)$. Therefore, the action of the Heisenberg Lie algebra is defined on the classes of differential forms on $\mathcal{C}_n$. We even propose extending this construction on forms by utilizing the Hochschild-Kostant-Rosenberg theorem. By identifying the graded rings $\Omega^{\ast}(\mathcal{C}_n)$, representing K\"ahler's differential forms on $\mathcal{C}_n$, and the Hochschild homology $HH_{\bullet}(\mathcal{O}(\mathcal{C}_n))$, associated with the coordinate ring $\mathcal{O}(\mathcal{C}_n)$, we aim to define an action of the Heisenberg Lie algebra (or potentially its loop Lie algebra) on the direct sum of Hochschild chain complexes of $\mathcal{O}(\mathcal{C}_n)$. Additionally, we intend to ensure that this action commutes with Hochschild differentials.
	To provide a description of the ring $\Omega^{\ast}(\mathcal{C}_n)$ or the Hochschild homology $HH_{\bullet}(\mathcal{O}(\mathcal{C}_n))$, it is necessary to establish an explicit presentation of the algebra $\mathcal{O}(\mathcal{C}_n)$ as in Theorem~\ref{mainThm} (or Conjecture~\ref{conjecture}, in general). 
	Having an explicit presentation of $\mathcal{O}(\mathcal{C}_n)$ also allows for the straightforward generation of elliptic curves and algebraic loops within $\mathcal{C}_n$. Our proposed representation suggests that $\mathcal{C}_n$ can be viewed as the product of the affine plane $\mathbb{A}_{\mathbb{C}}^2$ and $\mathcal{C}_n^0$, which is isomorphic to a subvariety of $\mathcal{C}_n$ comprising traceless matrix pairs. As a result, elliptic curves and loops can be realized on the affine plane $\mathbb{A}_{\mathbb{C}}^2$.
	
	Let us now examine how the presentation of the commuting variety ${\mathcal Com}_n$ can be used to uncover elliptic curves and algebraic loops on $\mathrm{Hilb}_n(\mathbb{C}^2)$.
	To this end, we need another description of $\mathcal{O}({\mathcal Com}_n)$ due to Gan and Ginzburg \cite{GG},  Domokos \cite{D} and Vaccarino \cite{Va}. 
	Let $U_n$ be the variety of pairs of diagonal $n\times n$ matrices.
	Since~$U_n$ can be identified with 
	the $n$-th product $(\mathbb{C}^2)^n$,  
	its coordinate ring is isomorphic to the ring $B_n\coloneqq\mathbb{C}[x,y]^{\otimes n}$. The symmetric group $S_n$ acts on $U_n$ by permutation of diagonal entries and the corresponding action on $B_n$ is given by permutation of factors. On the other hand, there is a group monomorphism $ \phi\colon S_n \to \mathrm{GL}_n$ given by the permutation representation, which induces  
	the action of $S_n$ on $\mathrm{Com}_n$. Clearly, the embedding of $U_n$ into $\mathrm{Com}_n$ is $S_n$-equivariant and we have a natural embedding $i\colon U_n\git S_n \hookrightarrow \mathrm{Com}_n\git S_n$. Composing $i$ with the projection   
	$\mathrm{Com}_n\git S_n \twoheadrightarrow \mathrm{Com}_n\git\mathrm{GL}_n$ yields a morphism of algebraic varieties  $\pi\colon U_n\git S_n \to {\mathcal Com_n}$. It has been proved in \cite{D,GG,Va} that $\pi$ is an isomorphism, or equivalently, the restriction homomorphism
	\begin{equation}
		\label{VacDom}
		\pi_{\ast} \colon \mathcal{O}({\mathcal Com}_n) \to B_n^{S_n} \,   
	\end{equation}
	is an algebra isomorphism.
	Note that, since the space $U_n\git S_n$ can be identified with the $n$-th symmetric power $ S^n \mathbb{C}^2\coloneqq (\mathbb{C}^2)^n\git S_n$, we get 
	the isomorphism
	\[
	S^n \mathbb{C}^2 \cong {\mathcal Com_n}
	\]
	defined by
	\begin{equation*}
		((\lambda_1,\mu_1), \dots ,(\lambda_n,\mu_n)\big) \mapsto 
		(\mathrm{Diag}[\lambda_1, \dots ,\lambda_n], \mathrm{Diag}[\mu_1, \dots ,\mu_n]) .
	\end{equation*}

	Recall that the \textit{Hilbert-Chow morphism} is the map $\rho\colon \mathrm{Hilb}_n(\mathbb{C}^2) \to  S^n \mathbb{C}^2$ defined as follows. Let $(X,Y,i) \in \mathrm{Hilb}_n(\mathbb{C}^2)$. Since $[X,Y]=0$, both matrices can be simultaneously transformed into upper triangular matrices with diagonal entries~$(\lambda_1, \dots ,\lambda_n)$ and $(\mu_1, \dots ,\mu_n)$ respectively. In this way, there is a well-defined morphism $\rho(X,Y,i)\coloneqq\big((\lambda_1,\mu_1), \dots ,(\lambda_n,\mu_n)\big)$. Thus, the Hilbert-Chow morphism can be defined as the map
	\begin{equation}
		\label{HCcom}
		\rho\colon  \mathrm{Hilb}_n(\mathbb{C}^2) \to {\mathcal Com_n}\, , \quad (X,Y,i) \mapsto  (X,Y) \, .
	\end{equation}
	
	Now, by selecting an elliptic curve or a loop on the $\mathbb{C}^2$ component of ${\mathcal Com_n}$ in \eqref{factor}, we can easily pull back the corresponding curve or loop through the mapping~$\rho$. Indeed, let us consider the elliptic curve $y^2=x^3+ax+b$ on $\mathbb{C}^2$. We define
	\[
	X=\mathrm{Diag}[\lambda_1x,\lambda_2x, \dots, \lambda_nx], \quad  Y=\mathrm{Diag}[0,\dots, 0, y],
	\]
	where $\lambda_1,\dots,\lambda_n$ are distinct and $\lambda_1+\dots+\lambda_n=1$. Clearly, this choice of $(X,Y)$ gives us an elliptic curve on ${\mathcal Com_n}$. For the vector $i=(1,1,\dots,1)^T \in \mathbb{C}^n$, it is evident that $i, Xi,X^2i,\dots, X^{n-1}i$ form a basis for $\mathbb{C}^n$.
	Therefore, for every point~$(x,y)$ on the elliptic curve, the triple $(X,Y,i)$ represents a point on $\mathrm{Hilb}_n(\mathbb{C}^2)$. Similarly, by selecting a loop $xy=1$, we can generate the corresponding loop on~$\mathrm{Hilb}_n(\mathbb{C}^2)$.
	
	Moreover, let $G$ be the unimodular affine Cremona group of $\mathbb{C}^2$, which is the normal subgroup of $\mathrm{Aut}_{\mathbb C}(\mathbb{C}[x,y])$ preserving the symplectic form $dx \wedge dy$. The group $G$ acts on $\mathcal{C}_n$, $\mathrm{Hilb}_n(\mathbb{C}^2)$, and ${\mathcal Com_n}$ as follows: for any $\sigma \in G$, $\sigma.(X,Y,i)=(\sigma^{-1}(X), \sigma^{-1}(Y))$, where we replace $(x,y)$ with the polynomials $\sigma^{-1}(x), \sigma^{-1}(y) \in \mathbb{C}[x,y]$ in the expression $(X,Y)$.
	This action gives rise to embeddings of $G$ into the automorphism groups of the varieties $\mathcal{C}_n$, $\mathrm{Hilb}_n(\mathbb{C}^2)$, and ${\mathcal Com_n}$. Therefore, under the action of $G$ on these varieties, elliptic curves and loops are transformed into elliptic curves and loops, respectively.

	\section{\texorpdfstring{The generating relation of $\CCC$}{The generating relation of CC4}}\label{s:relation}
	In this section we shall obtain an algebraic relation between the generators of~$\CCC$ listed above. It will be convenient to recollect some of the techniques we shall use.
	
	Firstly, note that identity~\eqref{vector-covector} provides a method of permuting variables by ``paying a small price''.
	Moreover, it implies that the morphism $\Phi \colon \C_n \to \C_n$, commonly referred as \emph{Fourier map}, induced by sending the pair of matrices $(X,Y)$ to~$(Y, -X)$, gives rise to an automorphism of $\CC[\C_4]$, explicitly defined on the generators by:
	\begin{equation}\label{Fourier}
		\begin{multlined}
			(a_1,a_2, a_3,a_4,a_5,a_6,a_7,a_8,a_9,a_{10},a_{11},a_{12},a_{13},a_{14}) \mapsto \\
			(a_2,-a_1, a_5,-a_4,a_3,a_9,-a_8,a_7,-a_6,a_{14},-a_{13},a_{12},-a_{11},a_{10})
		\end{multlined}
	\end{equation}
	
	In addition, we shall frequently use the Cayley-Hamilton theorem which states that for any traceless $M \in \M_4$,
	\begin{equation}\label{tracelessCayleyHamilton}
		M^4 = \frac12\Tr(M^2)M^2+\frac13\Tr(M^3)M-\frac18(\Tr^2(M^2)-2\Tr(M^4))I_4.	
	\end{equation}
	
	As an immediate consequence we obtain the following identities:
	\begin{align}\label{M^6}
		\Tr(M^6)={}&\frac34\Tr(M^4)\Tr(M^2)+\frac13\Tr^2(M^3)-\frac18\Tr^3(M^2), \\
		\Tr(A^4B^2)={}&\frac12\Tr(A^2)\Tr(A^2B^2)+\frac13\Tr(A^3)\Tr(AB^2)\nonumber\\
		&-\frac18(\Tr^2(A^2)-2\Tr(A^4))\Tr(B^2).
		\label{A^4B^2}
	\end{align}
	
	Using the strategies named above, we can obtain the following general result.
	\begin{lemma}\label{A^4B^2_} 	
		Let $(X,Y) \in \C_n$, then the following identities hold:
		\begin{align}
			\Tr(ABAB)=&\Tr(A^2B^2) + \frac12n(n-1) \label{ABAB-A^2B^2} \\
			\Tr(A^3BAB)=&\Tr(A^4B^2)+\frac{2n-3}2\Tr(A^2) \label{A^2BA^2B-A^4B^2} \\ 
			\Tr(A^2BA^2B)=&\Tr(A^4B^2)+(n-2)\Tr(A^2)		\label{A^3BAB-A^4B^2}
		\end{align}
	\end{lemma}
	
	\begin{proof}
		The first identity was already obtained in~\cite{N}, but we include it here for the sake of completion. Let us prove the second one.
		
		Consider a quadruple $(X, Y, v, w)$ as in~\eqref{vector-covector}. It is evident that the identity
		\begin{equation}\label{abvw}
			AB - BA + I_n = vw
		\end{equation}
		also holds, where $A$ and $B$ are the traceless versions of the matrices $X$ and $Y$, respectively. By taking traces we get that
		\[
		wv = n.
		\]
		Moreover, multiplying~\eqref{abvw} by $A^k$, we obtain that 
		\[
		\Tr(A^k)=wA^kv,
		\]
		for any integer $k$.
		In a similar fashion, we get
		\begin{equation*}
			A^3BAB-A^3B^2A+A^3B=A^3Bvw.
		\end{equation*}
		On one hand, taking traces on both sides yields:
		\begin{equation}\label{**}
			\Tr(A^3BAB)-\Tr(A^4B^2)+\Tr(A^3B)=wA^3Bv.
		\end{equation}
		On the other hand, we can operate as follows:
		\[
		A^3(AB+I_n-vw)B-A^3B^2A+A^3B=A^3Bvw,
		\] 
		which implies: 
		\begin{equation}\label{***}
			2\Tr(A^3B)-wBA^3v=wA^3Bv.
		\end{equation}
		Note that
		\begin{align*}
			wBA^3v&=w(AB+I_n-vw)A^2v\\
			&=wABA^2v+(1-n)wA^2v\\
			&=wA(AB+I_n-vw)Av+(1-n)wA^2v\\
			&=wA^2BAv+wA^2v-wAvwAv+(1-n)wA^2v\\
			&=wA^2BAv+(2-n)wA^2v\\ 
			&=wA^2(AB+I_n-vw)v+(2-n)wA^2v\\
			&=wA^3Bv+(3-2n)wA^2v\\
			&=wA^3Bv+(3-2n)\Tr(A^2).
		\end{align*}
		Therefore, the identity follows by equalising~\eqref{**} and~\eqref{***}.
		
		Similarly, multiplying~\eqref{abvw} by $A^2$ from the left side and by $AB$ from the right side, we obtain:
		\begin{equation}\label{*1}
			\Tr(A^3BAB)-\Tr(A^2BA^2B)+\Tr(A^3B)=wABA^2v.
		\end{equation}
		Note that
		\begin{equation*}
			\begin{aligned}
				wABA^2v&=wA(AB+I_n-vw)Av\\
				&=wA^2BAv+\Tr(A^2)\\
				&=wA^2(AB+I_n-vw)v+\Tr(A^2)\\
				&=wA^3Bv+wA^2v-wA^2vwv+\Tr(A^2)\\
				&=wA^3Bv+(2-n)\Tr(A^2)
			\end{aligned}
		\end{equation*}
		so using~\eqref{*1} together with~\eqref{A^3BAB-A^4B^2} we deduce the desired identity.
	\end{proof}

	\begin{proposition}\label{maineq}
		The polynomial
		\[
		r_1 = a_3(8+a_4^2-a_3a_5)-2(a_7^2-a_6a_8)+2(a_3a_{12}-2a_4a_{11}+a_5a_{10})
		\]
		defines a relation in the coordinate ring of $\C_4$.
	\end{proposition}
	
	\begin{proof}
		Let us consider the formal power expansion of $\Tr((A+tB)^6)$. The coefficient at $t^2$ is
		\[
		6\Tr(A^4B^2)+6\Tr(A^3BAB)+3\Tr(A^2BA^2B)
		\]
		which, making use of equations~\eqref{A^4B^2},~\eqref{A^2BA^2B-A^4B^2} and~\eqref{A^3BAB-A^4B^2}, together with Lemma~\ref{A^4B^2_}, is equal to
		\begin{equation}\label{r1_first}
			21a_3-\frac{5}{8}(3a_3^2a_5-8a_6a_8-12a_3a_{12}-6a_5a_{10}).
		\end{equation}
		
		On the other hand, applying the Cayley-Hamilton theorem~\eqref{M^6} first, we derive
		\begin{multline*}
			\Tr((A+tB)^6)=\frac34\Tr((A+tB)^4)\Tr((A+tB)^2)\\
			+\frac13\Tr^2((A+tB)^3)-\frac18\Tr^3((A+tB)^2).
		\end{multline*}
		After expanding the powers, gathering the coefficients, and further applying the already known substitutions together with equation~\eqref{ABAB-A^2B^2}, we spot that the coefficient at~$t^2$ can be written as 
		\begin{equation}\label{r1_second}
			-\frac{3}{8}  a_5 a_3^2-\frac{3}{2} a_4^2 a_3+\frac{3}{4} \left(6 a_{12}+12\right) a_3+3 a_7^2+2 a_6 a_8+\frac{3 a_5 a_{10}}{4}+6 a_4 a_{11}.
		\end{equation}
		
		Therefore, the expressions obtained in~\eqref{r1_first} and~\eqref{r1_second} must be equal in $\CCC$, resulting in the desired relation.
	\end{proof}

	
	\section{Full description}\label{section-main}
	
	In this section we shall focus on the Poisson structure of $\CCC$ which will allow us to produce new identities in a more algorithmic procedure, until we find the complete set of generators of the polynomial ideal of relations.
	
	We first want to recollect a set of identities obtained in~\cite[Lemmas~2.4 and~2.5]{N}, that will assist us to proceed forward.
	
	\begin{lemma}The following identities hold in $\CCC$:
		\begin{equation}
			\begin{aligned}
				\Tr(A^2BAB)={}&\Tr(A^3B^2),  \\
				\Tr(ABAB^2)={}&\Tr(A^2B^3),\\
				\Tr(A^2BAB^2)={}&\Tr(A^2B^2AB) - 8, \\
				\Tr((AB)^3)={}&\Tr(A^2B^2AB)+3a_4-4,\\
				\Tr(A^3B^3)={}&\Tr(A^2B^2AB)-2a_4-4, \label{listAsBs}\\
				\Tr(A^2B^3)={}&\frac1{12}( a_3a_9 +6a_4a_8 +3a_5a_7),\\
				\Tr(A^3B^2)={}& \frac1{12}(a_5a_6+6a_4a_7 +3a_3a_8),\\
				\Tr(A^3B^3)={}&\frac1{20}(-a_4^3+6a_7a_8+3a_5a_{11}+9a_4a_{12}+3a_3a_{13}+\frac23a_6a_9\\
				&-\frac32a_3a_4a_5-16a_4).
			\end{aligned}
		\end{equation}
		\noproof
	\end{lemma}

	A commutative $\CC$-algebra $R$ together with a Lie bracket $\{-,-\}$ is said to be a \emph{Poisson algebra} whenever the equation
	\[
	\{ a,bc \} = \{a,b\}c + b\{c,a\}
	\]
	is satisfied for all $a,b,c\in R$.
	
	Let $R = \CC\langle x, y \rangle$ be the free non-commutative associative $\CC$-algebra on two generators. It can be equipped with the bracket $\{-, -\}\colon R\times R\rightarrow R$ defined by 
	\begin{equation}
		\{u_1\cdots u_p, \ v_1\cdots v_q\}=\sum_{i=1}^{p}\sum_{j=1}^{q}\omega(u_i,v_j)u_{i+1}\cdots u_pu_1\cdots u_{i-1}v_{j+1}\cdots v_q v_1\cdots v_{j-1},
		\label{Kont}
	\end{equation}
	where each $u_1,\dots,u_p,v_1,\dots, v_q $ is either $x$ or $y$, and the bilinear form $\omega$ on $\CC x\oplus\CC y$ is defined by
	\[
	\omega(x,y)=-\omega(y,x)=1, \qquad \omega(x,x)=\omega(y,y)=0.
	\] 
	
	In~\cite{K} it is illustrated how this bracket endows $R$ with a Leibniz algebra structure and, furthermore, the following is known.
	
	\begin{proposition}[{\cite[Proposition 1.5]{G}}] 
		The linear space spanned by the elements of the form $ab - ba$ for all $a, b \in R$ is a central ideal of the Leibniz algebra $R$. Let us denote by $\LL$ the quotient of $R$ by this ideal. It is a Lie algebra and, furthermore, it induces a Poisson algebra structure on the free commutative algebra~$\Sym(\LL)$. \noproof
	\end{proposition}

	This is an example of a Necklace Lie algebra canonically associated to a quiver, which allows us to define the following map (see \cite{G}):

	\begin{equation*}
		\begin{aligned}
			tr\colon \mathcal{L}&\rightarrow \CC[(M_n\times M_n)\git \GL_n]\\
			x^{k_1}y^{l_1}\cdots x^{k_m}y^{l_m}&\mapsto \big( (X,Y)\mapsto  \Tr(X^{k_1}Y^{l_1}\cdots X^{k_m}Y^{l_m})\big)
		\end{aligned}
	\end{equation*}
	endowing $\CC[(M_n\times M_n)\git \GL_n]$ with a Lie and, consequently, a Poisson algebra structure. As mentioned in the Introduction, the innate Poisson algebra structures of $\CCCn$ and $\CCVn$ are, in fact, induced by this Poisson bracket on $\CC[(M_n\times M_n)\git \GL_n]$. Note that the number of $X$ and~$Y$ on each generator induces a bigrading on this polynomial ring.

	To proceed, we shall compute some of the brackets among the generators of the algebra of invariants~$\CC[(M_n\times M_n) \git \GL_n]$ modulo the Calogero-Moser relations. It is worth noting that while the Poisson structure of the Necklace Lie algebra can be easily computed on explicit generators, it becomes a more challenging task when generic traceless matrices are used.
	
	By a direct computation we can verify that $\{a_1, a_2\}=4$ and that the Poisson bracket of~$\Tr(X)$ with any of $\Tr(A^2),\Tr(AB),\Tr(B^2)$ is zero. Then, by induction we can immediately establish that $\{\Tr(X),\Tr(w)\}=0$ for any word $w$ involving the letters $A$ or $B$. Hence, all the Poisson brackets involving $a_1$ or $a_2$ are trivial. 
	
	On top of the skew-symmetry, there is some inner symmetry occurring within the Lie bracket: whenever $x$ and $y$ are swapped in~\eqref{Kont}, the result of the bracket is the additive inverse with $x$ and $y$ swapped correspondingly. Since the trace of any word is preserved under cyclic permutations, there is a well-defined involution~$(-)^s$, that fixes $a_4$ and $a_{12}$, and
	\begin{eqnarray*}\label{symm}
		a_3^s=a_5, \ a_6^s=a_9, \ a_7^s=a_8, \, a_{10}^s=a_{14}, \ a_{11}^s=a_{13},
	\end{eqnarray*}
	thus, consequently:
	\begin{equation}\label{symm.Poisson}
		\{a_i^s,a_j^s\}=-\{a_i,a_j\}^s.
	\end{equation}
	
	\begin{lemma}[\cite{N}]\label{pq}
		For any $p,q\geq 2$ the following holds:
		\[
		\Big\{\Tr(A^p),\ \Tr(B^q)\Big\}=pq\Tr(A^{p-1}B^{q-1})-\frac{pq}n\Tr(A^{p-1})\Tr(B^{q-1}).
		\]
		\noproof
	\end{lemma}
	
	This lemma allows us to directly compute the very first brackets:
	\begin{equation}
		\begin{aligned}\label{br1}
			\{a_3,a_5\} &=4a_4,\\
			\{a_3, a_9\}&=6a_8,\\
			\{a_3, a_{14}\}&=8a_{12},\\
			\{a_6, a_9\}&=9a_{12}-\frac94a_3a_5,\\
		\end{aligned}
	\end{equation}
	
	
	and, together with the corresponding substitutions from~\eqref{listAsBs}, we obtain
	\begin{equation} \label{br2}
		\begin{aligned}
			\{ a_6, a_{14}\}={}& 12\Tr(A^2B^3)-3\Tr(A^2)\Tr(B^3)\\
			={}&6a_4a_8+3a_5a_7-2a_3a_9,\\
			\{ a_{10}, a_{14}\}={}& 16\Tr(A^3B^3)-4\Tr(A^3)\Tr(B^3)\\
			={}&-\frac45a_4(16+\frac32a_3a_5+a_4^2)+\frac4{15}(18a_7a_8-13a_6a_9)\\
			&+\frac{12}5(a_5a_{11}+3a_4a_{12}+a_3a_{13}).\\
		\end{aligned}
	\end{equation}

	Our next strategy is to substitute the values of $A$ and $B$ in terms of $X$ and $Y$, compute the Poisson bracket using~\eqref{Kont} and, finally transform the obtained expressions back in terms of $A$ and $B$. For instance,
	\begin{equation}\label{br3.0}
		\begin{aligned}
			\{a_3,a_4\}&=\{\Tr(A^2),\Tr(AB)\}\\
			&=\{\Tr((X-\frac14\Tr(X)I_4)^2),\Tr((X-\frac14\Tr(X)I_4)(Y-\frac14\Tr(Y)I_4))\}\\
			&=\{\Tr(X^2)-\frac14\Tr^2(X), \Tr(XY)-\frac14\Tr(X)\Tr(Y)\}\\
			&=2\Tr(X^2)-\frac12\Tr^2(X)-\frac12\Tr^2(X)+\frac12\Tr^2(X)\\
			&=2(\Tr(X^2)-\frac14\Tr^2(X))=2\Tr(A^2)=2a_3.
		\end{aligned}
	\end{equation}
	
	In this way, we may obtain the following brackets:
	\begin{equation}\label{br3}
		\begin{aligned}
			\{a_3,a_7\}&=2a_6,\\
			\{a_3,a_8\}&=4a_7,\\
			\{a_3,a_{11}\}&=2a_{10},\\
			\{a_3,a_{12}\}&=4a_{11},\\
			\{a_3,a_{13}\}&=4\Tr(A^2B^2)+2\Tr(ABAB)=6a_{12}+12.
		\end{aligned}
	\end{equation}
	
	After a careful analysis of~\eqref{Kont} one can detect that the element $a_4$ is semi-simple and instantly obtain:  
	\begin{equation}\label{br4}
		\begin{aligned}    
			\{a_4, a_6\}&=-3a_6,\\
			\{a_4, a_7\}&=-a_7,\\
			\{a_4, a_{10}\}&=-4a_{10},\\
			\{a_4, a_{11}\}&=-2a_{11},\\
			\{a_4, a_{12}\}&=0.\\
		\end{aligned}
	\end{equation}

	Mixing all the approaches used above, together with~\eqref{ABAB-A^2B^2}--\eqref{listAsBs} we can assemble the following brackets:
	
	\begin{equation}\label{br5}
		\begin{aligned}
			\{a_6, a_7\}={}&3a_{10}-\frac34a_3^2,  \\
			\{a_6, a_8\}={}&6a_{11}-\frac32a_3a_4, \\ 		
			\{a_6, a_{11}\}={}&\frac74a_3a_6, \\
			\{a_6, a_{12}\}={}&2a_4a_6+\frac32a_3a_7,  \\
			\{a_6, a_{13}\}={}&\frac34a_5a_6+\frac92a_4a_7,  \\
			\{a_7, a_8\}={}&3a_{12}+12-a_4^2+\frac14a_3a_5,  \\ 
			\{a_7, a_{10}\}={}&-\frac73a_3a_6,  \\ 
			\{a_7, a_{11}\}={}&-\frac56a_4a_6+\frac14a_3a_7, \\
			\{a_7, a_{12}\}={}&a_3a_8+\frac16a_5a_6,  \\
			\{a_7, a_{13}\}={}&\frac23a_3a_9+a_4a_8+\frac54a_5a_7,  \\
			\{a_7, a_{14}\}={}&4a_5a_8+\frac23a_4a_9, \\
			\{a_{10}, a_{11}\}={}&-\frac12a_3^3+\frac13a_6^2+3a_3a_{10}, \\
			\{a_{10}, a_{12}\}={}&-a_3^2a_4+\frac23a_6a_7+4a_3a_{11}+2a_4a_{10}, \\
			\{a_{10}, a_{13}\}={}&10a_3-\frac32a_3^2a_5+a_6a_8+6a_3a_{12}+3a_5a_{10} \\
			\{a_{11}, a_{12}\}={}&9a_3-\frac12a_3^2a_5+\frac{11}6a_6a_8-\frac32a_7^2+2a_3a_{12}+a_5a_{10}. \\
			\{a_{11}, a_{13}\}={}&\frac15a_4(53-2a_4^2-3a_3a_5)+\frac65(3a_4a_{12}+a_3a_{13}+a_5a_{11}) \\
			&+\frac1{60}(9a_7a_8+31a_6a_9), 
		\end{aligned}
	\end{equation}
	
	In conclusion, equations~\eqref{br1}--\eqref{br5} together with the involution determined in~\eqref{symm.Poisson} complete the explicit multiplication table of the bracket on generators~\eqref{generators}. Note that while this bracket induces a non-associative antisymmetic algebra structure on the polynomial ring~$\CC[a_1, \dots, a_{14}]$, it is far from being a well-defined Poisson algebra structure. 
	Indeed, a computation of the following Jacobiator asserts that
	\[
	\{\{a_5,a_{10}\},a_{12}\}+\{\{a_{12},a_5\},a_{10}\}+\{\{a_{10},a_{12}\},a_5\}=8r_1,
	\]
	where $r_1$ is the relation obtained in Proposition~\ref{maineq}. Since we know that this Jacobiator must vanish in $\mathbb{C}[\mathcal{C}_4]$, this presents us an alternative way of proving Proposition~\ref{maineq}. To conclude the proof of the main result, we claim this relation is sufficient to entirely describe the coordinate ring $\CCC$. We also need the following remark.

	\begin{remark}\label{rem:free}
		The algebraic structure of both~$\CCC$ and~$\CCV$ is studied in~\cite{EG} (see in particular Propositions 4.15 (ii) and~ 6.10 (i)), where it is proven that they
		are free modules over ${\mathbb{C}[x_1,\dots,x_n]^{S_n}\otimes \mathbb{C}[y_1,\dots,y_n]^{S_n}}$ of rank $n!$.
		Since the Hilbert series of $\mathbb{C}[x_1,\dots,x_n]^{S_n}$ is known to be 
		\[
		\frac{1}{(1-t)(1-t^2)\cdots(1-t^n)}
		\]
		the Hilbert series of both $\CC[\mathcal{C}_n]$ and~$\CC[{\mathcal Com_n}]$ must be a rational function with the numerator consisting of a polynomial with positive coefficients adding up to $n!$ and the denominator being equal to $(1-t)^2(1-t^2)^2\cdots(1-t^n)^2$.
	\end{remark}

	\begin{theorem}\label{main_theorem}
		There is a Poisson algebra isomorphism 
		\[
		\CCC \cong \CC[a_1,a_2]\otimes \dfrac{\CC[a_3,\dots, a_{14}]}{I},
		\] 
		where $I$ is the non-associative ideal generated by $r_1$, and it is equal to the  polynomial ideal generated by the relations listed in Appendix~\ref{ap:CM}. 
		Moreover, the generators of~$\CCC$ as a free $\ring$-module are listed in Appendix~\ref{ap:free}.
	\end{theorem}
	
	\begin{proof}
		The outline of the proof is as follows: we shall strategically obtain a set of relations by using the Poisson structure and the Fourier map~\eqref{Fourier} belonging to the non-associative ideal generated by $r_1$. Then, we shall define an injective morphism of $\ring$-modules from the free module generated by the monomials listed in Appendix~\ref{ap:free} to the canonical module structure of~$\CC[a_1,a_2]\otimes \CC[a_3,\dots, a_{14}]/I$ and verify that their Hilbert series coincide.  Therefore, since $\CC[a_1,a_2]\otimes \CC[a_3,\dots, a_{14}]/I$ sits inside~$\CCC$ and they are both free of the same rank, they are isomorphic.
		
		First, let us introduce the following polynomials (whose exact expressions are written in the Appendix~\ref{ap:CM}): 
		\begin{equation*}
			\begin{aligned}
				r_2&=-6\{r_1,a_7\}, \\
				r_3&= \frac16\{r_2,a_5\}, \\
				r_4&=\frac12\{\{r_1, a_8\},a_6\}, \\
				r_5&=\frac12\{\{r_1, a_8\},a_9\},\\
				r_6&=3\{r_1,a_8\}.
			\end{aligned}
		\end{equation*}
		Applying the Fourier map, we automatically get $t_1, \dots, t_6$, respectively. Finally, we obtain
		\begin{equation*}
			\begin{aligned}
				s_1&=\frac14\{r_1,a_5\},\\
				s_2&= a_3t_1-\frac16\{r_{5}, a_3\},\\
				s_{3}&=  -\frac19\{r_2,a_9\}.
			\end{aligned}
		\end{equation*} 
		which are invariant under the Fourier map. Note that the coefficients placed in front of the brackets are there just to ease the notation of each corresponding relation.
		
		Let $I$ be the polynomial ideal generated by $r_1, \dots, r_5, t_1, \dots, t_4, s_1, \dots, s_3$. With the help of a computer algebra system as \texttt{Macaulay2}~\cite{Mac2} we find that 
		\begin{equation}\label{eq:Hilbert}
			\frac{1+T^{2}+2\,T^{3}+4\,T^{4}+2\,T^{5}+4\,T^{6}+2\,T^{7}+4\,T^{8}+2\,T^{9}+T^{10}+T^{12}}{\left(1-T^{4}\right)^{2}\left(1-T^{3}\right)^{2}\left(1-T^{2}\right)^{2}(1-T)^2}
		\end{equation}
		is the Hilbert series of $\CC[a_1, \dots, a_{14}]/I$, graded by the sum of its natural bigrading.

		We know that~$\CCC$ is a free $\ring$-module of rank~$24$, as mentioned in Remark~\ref{rem:free}. We denote by~$M$ the free~$\ring$-module generated by the elements of Appendix~\ref{ap:free}. 
		There is a canonical map defined by the composition
		\[
		M \to \CC[a_1, \dots, a_{14}] \to \CC[a_1,a_2]\otimes \dfrac{\CC[a_3,\dots, a_{14}]}{I}
		\]
		To verify that it is injective, we compute a Gröbner basis with respect to the weighted degree reverse lexicographical order of the ideal generated by $I$ together with $a_3, a_5, a_6, a_9, a_{10}, a_{14}$. Then, it is a matter of checking that the basis of generators listed in Appendix~\ref{ap:free} are normal forms with respect to the Gröbner basis.	
		Finally, an examination of the degrees shows that the Hilbert series of~$M$ is exactly the same as~\eqref{eq:Hilbert}.
	\end{proof}
	
	\begin{remark}
		Note that in the proof we compute three polynomials that are not used: $r_6$, $t_5$, $t_6$. The reason we did this is because they will be needed in the following section, and as a curiosity they also appear when computing a Gröbner basis of the ideal $I$ with respect to the weighted degree reverse lexicographical order.
		As a polynomial ideal, it is easy to check  that none of the $12$ polynomials generating $I$ can be removed. Moreover, $I$ forms a prime ideal inside $\CC[a_1, \dots, a_{14}]$ defining an algebraic variety of dimension $8$. 
	\end{remark}
	
	\begin{remark} 
		As a consequence of Theorem~\ref{main_theorem}, we can confirm that the Poisson bracket defined on generators~\eqref{br1}--\eqref{br5} modulo the ideal $I$ is a well-defined Poisson bracket on $\CCC$.
	\end{remark}

	\section{Invariant commuting variety}
	
	The full description of $\CCC$ grants us with the full description of the coordinate ring of the invariant commuting variety of $4 \times 4$~matrices, as an almost immediate consequence.
	When we obtained the equations in Section~\ref{s:relation}, whenever we wanted to commute a pair of matrices, a ``small price'' had to be paid, that manifests as a term of a lower degree. In $\CCV$, these extra terms do not occur.
	
	\begin{theorem}
		The coordinate ring of the commuting variety of $4 \times 4$~matrices is isomorphic to $\CC[a_1, a_2] \otimes \CC[a_3, \dots, a_{14}]/J$, where $J$ is the ideal generated by the polynomials listed in the Appendix~\ref{ap:CV}.
	\end{theorem}
	
	\begin{proof}
		The method we use to find the relations is the following: We compute a Gröbner basis of $I$ (from the Theorem~\ref{main_theorem}) with respect to the weighted degree reverse lexicographical order, obtaining the $15$ polynomials listed in Appendix~\ref{ap:CM}. Then, for each non-homogeneous polynomial we remove the lower degree terms, obtaining a set of homogeneous polynomials, that are listed in Appendix~\ref{ap:CV}.
		
		It is a straightforward computation to check that all the polynomials obtained are relations in the coordinate ring of the invariant commuting variety. Then, to finish we replicate the final part of the proof of Theorem~\ref{main_theorem}.
	\end{proof}
	
	\section{On the Grothendieck Ring}
	
	In this section we compute the class in the Grothendieck ring of both the Calogero-Moser space and the invariant commuting variety. To do so, we will outline the computations done for the invariant commuting variety, which can be done in a similar fashion for the Calogero-Moser space.
	
	Recall that the Grothendieck ring in the category of quasi-projective complex varieties is the abelian group generated by isomorphism classes of quasi-projective varieties $[V]$ modulo the relation $[V] = [V_1] + [V_2]$, whenever $V$ can be decomposed as a disjoint union of a closed set $V_1$ and an open set $V_2$. Note that it forms a ring with its naturally defined product $[U] \cdot [V] = [U \times V]$. We denote the \emph{Lefschetz object} by $\LLL = \CC$.
	
	\begin{theorem}
		The class of the invariant commuting variety of a pair of $4\times 4$ matrices in the Grothendieck ring is $\LLL^8$.
	\end{theorem}
	\begin{proof}
		One of the standard ways to compute this invariant is to proceed by divide and conquer. Firstly, since $\mathbb{C}[{\mathcal Com_4}] = \CC[a_1, a_2] \otimes \CC[a_3, \dots, a_{14}]/J$, it is clear that 
		\[
		[{\mathcal Com_4}] = \LLL^2 \cdot [\Var(J)]
		\]
		
		Let us denote by $\widetilde{r_i}$, $\widetilde{s_i}$ and $\widetilde{t_i}$ the corresponding identities generating $J$ from Appendix~\ref{ap:CV}. We make the following transformations in order to ease the computations:
		\begin{align*}
			\widetilde{r_4}' &= -\dfrac{1}{2}(\widetilde{r_4} + 3a_3\widetilde{r_1}) \\
			\widetilde{r_5}' &= \widetilde{r_5} + 3a_4\widetilde{t_1} \\
			\widetilde{r_6}' &= \dfrac{1}{2}(\widetilde{r_3} + \widetilde{r_6}) \\
			\widetilde{t_6}' &= \dfrac{1}{2}(\widetilde{t_6} - 7\widetilde{t_3}) \\
			\widetilde{s_2}' &= -\dfrac{1}{2}(\widetilde{s_2} + \widetilde{s_3})
		\end{align*}
		Assisted by Gröbner bases it is immediate to check that, substituting the \emph{prime} versions of the relations by the original ones, we still generate the same ideal. 
		Solving directly the equations for the variables $a_5$, $a_8$, $a_9$, $a_{12}$, $a_{13}$, $a_{14}$, we get a complete solution where the values of $a_4$, $a_7$ and $a_{11}$ are free, and the following polynomial
		\[
		{w_1} = 288 a_{10}^3 - 288 a_{10}^2 a_3^2 + 90 a_{10} a_3^4 - 9 a_3^6 - 144 a_{10} a_3 a_6^2 + 68 a_3^3 a_6^2 + 24 a_6^4
		\]
		appears in the denominators of the solutions. Note that this polynomial involves only the variables $a_3$, $a_6$ and $a_{10}$, which correspond to the traces of the powers of the traceless form of the first matrix $X$ in the pair.		
		Furthermore, it is exactly 
		\[
		72 \prod_{1\leq i< j\leq 4} (\lambda_i-\lambda_j)^2,
		\]
		where $\lambda_1,\dots, \lambda_4$ are the eigenvalues of~$X$.
		
		We denote by $V_{{w_1}}$ the subvariety $\Var(J \cup \{{w_1}\})$ and by $U_{w_1}$ the open subset of~$\Var(J)$ where ${w_1} \neq 0$ (i.e., $X$ admits a simple spectrum), then
		\begin{equation*}
			[U_{w_1}] = \LLL^3 \cdot [\{ (a_3, a_6, a_{10}) \in \CC^3 \mid {w_1} \neq 0 \}].
		\end{equation*}
		
		In order to compute the class of $V_{w_1}$, we add ${w_1}$ to the set of equations, and we can now solve directly the system whenever $a_6 (-4a_{10} a_3 + a_3^3 + 4 a_6^2) \neq 0$. Let us denote:
		\begin{align*}
			w_{21} &= a_6 \\
			w_{22} &= -4a_{10} a_3 + a_3^3 + 4 a_6^2.
		\end{align*}
		Note that, as in $w_1$, only the variables $a_3$, $a_6$ and $a_{10}$ appear in the polynomials~$w_{21}$ and $w_{22}$.
		
		Therefore, $V_{w_1}$ can be decomposed as the union of $V_{1}$, $V_{21}$, $V_{22}$, $V_{23}$, where
		\begin{align*}
			&V_{1} \text{ is the open subset of }V_{w_1} \text{ where } w_{21}w_{22} \neq 0, \\
			&V_{2} \text{ is the subset of }V_{w_1} \text{ where } w_{21} = 0 \text{ and } w_{22} \neq 0, \\
			&V_{3} \text{ is the subset of }V_{w_1} \text{ where } w_{21} \neq 0\text{ and } w_{22} = 0, \\
			&V_{4} \text{ is the closed subset of }V_{w_1} \text{ where } w_{21} = 0\text{ and } w_{22} = 0.
		\end{align*}
		Solving directly the equations, we can find the following:
		\begin{itemize}
			\item In $V_1$, the variables $a_4$, $a_8$ and $a_{11}$ are free, while $a_5$, $a_7$, $a_9$, $a_{12}$, $a_{13}$, $a_{14}$ are completely determined.
			\item In $V_{2}$, the variables $a_4$, $a_5$ and $a_{7}$ are free, while the variables $a_{8}$, $a_{9}$, $a_{11}$, $a_{12}$, $a_{13}$, $a_{14}$ are completely determined.
			\item In $V_{3}$, the variables $a_4$, $a_8$ and $a_{11}$ are free, while the variables $a_{5}$, $a_{7}$, $a_{11}$, $a_{12}$, $a_{13}$, $a_{14}$ are completely determined.
			\item In $V_{4}$, the variables $a_4$, $a_5$ and $a_{8}$ are free, while the variables $a_{7}$, $a_{9}$, $a_{11}$, $a_{12}$, $a_{13}$, $a_{14}$ are completely determined.		
		\end{itemize}
		Therefore, 
		\begin{align*}
			[V_{1}]  &= \LLL^3 \cdot [\{(a_3, a_6, a_{10}) \in \CC^3 \mid w_1 = 0, w_{21}w_{22} \neq 0 \}], \\
			[V_{2}]  &= \LLL^3 \cdot [\{(a_3, a_6, a_{10}) \in \CC^3 \mid w_1 = 0, w_{21} = 0, w_{22} \neq 0 \}], \\
			[V_{3}]  &= \LLL^3 \cdot [\{(a_3, a_6, a_{10}) \in \CC^3 \mid w_1 = 0, w_{21} \neq 0, w_{22} = 0 \}], \\
			[V_{4}]  &= \LLL^3 \cdot [\{(a_3, a_6, a_{10}) \in \CC^3 \mid w_1 = 0, w_{21} = 0, w_{22} = 0 \}].
		\end{align*}
		Factoring out in
		\[
		[\Var(J)] = [U_{w_1}] + [V_1] + [V_2] + [V_3] + [V_4],
		\]
		we realize that the sum of all the classes left to find is exactly the class of the union of a decomposition of $\CC^3$, obtaining
		\[
		[\Var(J)] = \LLL^3 \cdot \LLL^3 = \LLL^6,
		\]
		and, consequently, $[{\mathcal Com_4}] = \LLL^8$.
	\end{proof}

	\begin{theorem} 
		The class of the fourth Calogero-Moser space  in the Grothendieck ring is $\LLL^8 - \LLL^7 + 2 \LLL^6-\LLL^5$. \noproof
	\end{theorem}		
	
	Doing the similar but much easier computations for matrix sizes two and three, we observe the pattern that allows us to effectively state the following conjecture:
	\begin{conjecture}
		The classes of the Calogero-Moser spaces and the invariant commuting varieties in the Grothendieck ring are the following:
		\begin{align*}
			[\mathcal{C}_n] &= b_{4n}\, \LLL^{2n} - b_{4n-2}\,\LLL^{2n-1} +\cdots+ (-1)^nb_{2n}\,\LLL^{n} \\
			[{\mathcal Com_n}] &= \LLL^{2n} ,
		\end{align*}
		where $b_i$ is the $i$-th Betti number of the Borel-Moore homology of $\mathcal{C}_n$.
	\end{conjecture}

	\section*{Acknowledgments}
	We would like to thank Ivan Shestakov, Efim Zelmanov, Vyacheslav Futorny, Claudio Procesi and Yura Berest
	for their valuable discussions. We would also like to thank Ulrich Thiel for pointing us their algorithmic approach developed in~\cite{BT}. The authors are grateful to the anonymous referee for their valuable comments, which have helped improve the manuscript.
	The second author would like to thank the V.~I.~Romanovskiy Institute of Mathematics for their kind hospitality during his stays in Tashkent.

	\appendix
	\section{List of equations} \label{appendix}
	\subsection{Calogero-Moser space}\label{ap:CM}
	
	The following 12 polynomials are the defining relations for the coordinate ring of the fourth Calogero-Moser space: 
	\begin{align*}
		r_1={}& 8a_3+a_3(a_4^2-a_3a_5)-2(a_7^2-a_6a_8)+2(a_3a_{12}-2a_4a_{11}+a_5a_{10})\\
		r_2={}& 48a_6-a_{6}(4a_{4}^2-a_{3}a_{5})-3a_{3}(a_3a_{8}-2a_{4}a_{7})+12(a_{6}a_{12}-2a_{7}a_{11}+a_{8}a_{10})\\
		r_3={}& -2a_{4}a_{5}a_{6}+3a_{3}a_{5}a_{7}-a_{3}^2a_{9}+4(2a_{6}a_{13}-3a_{7}a_{12}+a_{9}a_{10})\\
		r_4={}& -6a_3^2+24a_{10}-6a_{10}(a_{4}^2+a_{3}a_{5})-3a_{3}^2(a_{4}^2-a_{3}a_{5})\\
		& +2a_6(a_{3}a_{8}-2a_{4}a_{7}+a_{5}a_{6}) -12a_{3}(a_3a_{12}-2a_{4}a_{11})+24(a_{10}a_{12}-a_{11}^2)\\
		r_5={}& 12a_{4}a_{5}-48a_{13}+3a_4a_5(a_{3}a_{5}-a_{4}^2)+2(a_{3}a_{8}a_{9}-2a_{4}a_{7}a_{9}+a_{5}a_{6}a_{9})\\
		&-3a_4(a_{5}a_{12}-4a_4a_{13}+3a_3a_{14})+3a_{5}(a_{3}a_{13}-a_{5}a_{11})+ 12(a_{11}a_{14}-a_{12}a_{13})\\
		\empty\\
		t_1={}&8a_5+a_5(a_4^2-a_3a_5)-2(a_8^2-a_7a_9)+2(a_5a_{12}-2a_4a_{13}+a_3a_{14})\\
		t_2={}&48a_9-a_{9}(4a_{4}^2-a_{3}a_{5})-3a_{5}(a_5a_{7}-2a_{4}a_{8})+12(a_{9}a_{12}-2a_{8}a_{13}+a_{7}a_{14})\\
		t_3={}& -2a_{3}a_{4}a_{9}+3a_{3}a_{5}a_{8}-a_{5}^2a_{6}+4(2a_{9}a_{11}-3a_{8}a_{12}+a_{6}a_{14})\\
		t_4={}& -6a_5^2+24a_{14}-6a_{14}(a_{4}^2+a_{3}a_{5})-3a_{5}^2(a_{4}^2-a_{3}a_{5})\\
		&+2a_9(a_{5}a_{7}-2a_{4}a_{8}+a_{3}a_{9})  -12a_{5}(a_5a_{12}-2a_{4}a_{13})+24(a_{12}a_{14}-a_{13}^2)\\
		\empty\\
		s_1={}&-4a_{4}+a_4(a_{4}^2-a_{3}a_{5})+(a_{6}a_{9}-a_{7}a_{8})+2(a_{3}a_{13}-2a_{4}a_{12}+a_{5}a_{11})\\
		s_2={}& -96+2(9a_3a_5+16a_4^2)-72a_{12}+2a_4^2(a_3a_5-a_4^2)-a_3(a_3a_{14}-2a_5a_{12})\\
		&+2(a_3a_7a_9-2a_4a_7a_8+a_5a_6a_8)-a_5(a_5a_{10}-2a_3a_{12})\\
		&-2a_4(3a_3a_{13}-5a_4a_{12}+3a_5a_{11})+4(a_{10}a_{14}+2a_{11}a_{13}-3a_{12}^2)\\
		s_3={}& 12a_3a_5-48a_{12}+2a_{12}(2a_4^2+a_3a_5)+a_3(a_3a_{14}-4a_4a_{13})\\
		&+a_5(a_5a_{10}-4a_4a_{11})-4(a_{10}a_{14}-4a_{11}a_{13}+3a_{12}^2) 
	\end{align*}
	The extra three polynomials:
	\begin{align*}
		r_6={}&48a_7-3a_{7}(4a_{4}^2+a_{3}a_{5})+4a_{4}a_{5}a_{6}+18a_{3}a_{4}a_{8}-7a_{3}^2a_{9}\\
		&+12(a_{9}a_{10}-2a_{8}a_{11}+a_{7}a_{12}) \\
		t_5={}& 12a_{3}a_{4}-48a_{11}+3a_3a_4(a_{3}a_{5}-a_{4}^2)+2(a_{5}a_{6}a_{7}-2a_{4}a_{6}a_{8}+a_{3}a_{6}a_{9})\\
		&-3a_4(a_{3}a_{12}-4a_4a_{11}+3a_5a_{10})+3a_{3}(a_{5}a_{11}-a_{3}a_{13})+12(a_{10}a_{13}-a_{11}a_{12})\\
		t_6={}& 48a_8-3a_{8}(4a_{4}^2+a_{3}a_{5})+4a_{3}a_{4}a_{9}+18a_{4}a_{5}a_{7}-7a_{5}^2a_{6} \\
		& +12(a_{6}a_{14}-2a_{7}a_{13}+a_{8}a_{12}).
	\end{align*}

	\subsection{Generators of the free \texorpdfstring{$\ring$-module}{Generators of the free module}}\label{ap:free} \hfill
	\begin{align*}
		&1, \quad  
		a_{4}, \quad 
		a_{7}, \quad a_{8}, \quad 
		a_{4}^{2}, \quad a_{11}, \quad a_{12}, \quad a_{13}, \quad 
		a_{4}a_{7}, \quad a_{4}a_{8}, \quad 
		a_{4}^{3}, \quad a_{4}a_{11}, \quad a_{4}a_{12}, \\ &a_{4}a_{13}, \quad 
		a_{4}^{2}a_{7}, \quad a_{4}^{2}a_{8}, \quad 
		a_{4}^{4}, \quad a_{4}^{2}a_{11}, \quad a_{4}^{2}a_{12}, \quad a_{4}^{2}a_{13}, \quad 
		a_{4}^{3}a_{7}, \quad a_{4}^{3}a_{8}, \quad 
		a_{4}^{5}, \quad 
		a_{4}^{6}
	\end{align*}
	
	\subsection{Commuting variety}\label{ap:CV}
	\begin{align*}\label{relations_comm_var}
		&r_1-8a_3,&  &r_2-48a_{6},&  &r_3,&  &r_4+6a_{3}^2 -24a_{10},&  &r_5-12a_{4}a_{5}+48a_{13},&  &r_6-48a_{7},&  \\
		&t_1-8a_{5},&  &t_2-48a_{9},&  &t_3,&  &t_4+6a_{5}^2-24a_{14},&  &t_5-12a_{3}a_{4}+48a_{11},&  &t_6-48a_{8},&
	\end{align*}
	\[
	s_1+4a_{4}, \
	s_2-18a_{3}a_{5}-32a_{4}^2+72a_{12}+96, \
	s_3-12a_{3}a_{5}+48a_{12}.
	\]

\end{document}